\documentclass[12pt]{article}
\usepackage{amssymb,amsmath,amsthm, amsfonts}
\usepackage{graphicx}
\usepackage{epsfig}
\usepackage{tikz}

\def\disp{\displaystyle}

\def\dref#1{(\ref{#1})}

\theoremstyle{plain}
\newtheorem{theorem}{Theorem}[section]
\newtheorem{lemma}{Lemma}[section]

\theoremstyle{definition}

\newtheorem{remark}{Remark}[section]

\numberwithin{equation}{section}

\begin{document}

\title{\bf
Global solvability and boundedness in the $N$-dimensional quasilinear chemotaxis model with  logistic source and consumption of chemoattractant}

\author{
Jiashan Zheng 
\thanks{Corresponding author.   E-mail address:
 zhengjiashan2008@163.com (J.Zheng)}
\\
    School of Mathematics and Statistics Science,\\
     Ludong University, Yantai 264025,  P.R.China \\
}
\date{}

\maketitle \vspace{0.3cm}
\noindent
\begin{abstract}
We consider the following  chemotaxis model
$$
 \left\{\begin{array}{ll}
  u_t=\nabla\cdot(D(u)\nabla u)-\chi\nabla\cdot(u\nabla v)+\mu (u-u^2),\quad
x\in \Omega, t>0,\\
 \disp{v_t-\Delta v=-uv },\quad
x\in \Omega, t>0,\\
 \disp{(\nabla D(u)-\chi u\cdot \nabla v)\cdot \nu=\frac{\partial v}{\partial\nu}=0},\quad
x\in \partial\Omega, t>0,\\
\disp{u(x,0)=u_0(x)},\quad  v(x,0)=v_0(x),~~
x\in \Omega\\
 \end{array}\right.
$$
on  a  bounded domain
 $\Omega\subset\mathbb{R}^N(N\geq1)$,
 with smooth boundary $\partial\Omega, \chi$ and $\mu$ are  positive constants.
 Besides appropriate smoothness assumptions, in this paper it is only required
that $D(u)\geq C_{D}(u+1)^{m-1}$
for all $u\geq 0$ with some $C_{D} > 0$ and some
$$
 m>\left\{\begin{array}{ll}
1-\frac{\mu}{\chi[1+\lambda_{0}\|v_0\|_{L^\infty(\Omega)}2^{3}]}~~\mbox{if}~~
N\leq2,\\
 1~~~~~~\mbox{if}~~ N\geq3,\\
 \end{array}\right.
$$
then
for any sufficiently smooth initial data there exists a classical solution which is global in time and bounded, where $\lambda_{0}$ is  a positive constant which is corresponding to the maximal sobolev
regularity.  The results of this paper extends the results of Jin (J. Diff. Eqns.,
263(9)(2017), 5759--5772), who proved the possibility of boundness of weak solutions, in the case
  $m>1$ and  $N=3$.
\end{abstract}

\vspace{0.3cm}


\newpage
\section{Introduction}
Due to its important applications in biological and medical sciences, chemotaxis
research has become one of the most hottest topics in applied mathematics nowadays and
tremendous theoretical progresses have been made in the past few decades.
This paper
is devoted to making further development for the following quasilinear   chemotaxis systems with  logistic source and consumption of chemoattractant, reading as
\begin{equation}
 \left\{\begin{array}{ll}
  u_t=\nabla\cdot( D(u)\nabla u)-\chi\nabla\cdot(u\nabla v)+\mu(u-u^2),\quad
x\in \Omega, t>0,\\
 \disp{ v_t=\Delta v - uv},\quad
x\in \Omega, t>0,\\
 \disp{(\nabla D(u)-\chi u\cdot \nabla v)\cdot \nu=\frac{\partial v}{\partial\nu}=0},\quad
x\in \partial\Omega, t>0,\\
\disp{u(x,0)=u_0(x)},\quad  v(x,0)=v_0(x),~~
x\in \Omega,\\
 \end{array}\right.\label{1.ssderrfff1}
\end{equation}
where 
$\Omega\subset\mathbb{R}^N(N\geq1)$ is a bounded  domain 
 with smooth boundary
$\partial\Omega$, 
 $\Delta=\disp{\sum_{i=1}^N\frac{\partial^2}{\partial x^2_i}}$, $\disp\frac{\partial}{\partial\nu}$ denotes the outward normal
derivative on $\partial\Omega$,
$\chi>0$ is a parameter referred to as chemosensitivity,  $\mu u(1-u)(\mu >0)$ and $-vu$ are the proliferation or death of bacteria according to a generalized logistic law and  the consumption of chemoattractant, respectively. Here $u:=u(x,t)$ and $v:=v(x,t)$ denotes the density of the cells population and the concentration of the chemoattractant, respectively.
The nonlinear
nonnegative function $D(u)$ satisfies
\begin{equation}\label{91derfff61}
D\in  C^2([0,\infty))
\end{equation}
and
\begin{equation}\label{9162}
D(u) \geq (u+1)^{m-1}~ \mbox{for all}~ u\geq0
\end{equation}
with 
$m \in \mathbb{R}$.

In order to better understand model \dref{1.ssderrfff1}, we can see some previous contributions in this
direction. Assuming that $\mu\equiv0,$
 the  chemotaxis model \dref{1.ssderrfff1} can be reduced to quasilinear  chemotaxis model with consumption of chemoattractant
\begin{equation}
 \left\{\begin{array}{ll}
  u_t=\nabla\cdot( D(u)\nabla u)-\chi\nabla\cdot(u\nabla v),\quad
x\in \Omega, t>0,\\
 \disp{ v_t=\Delta v - uv},\quad
x\in \Omega, t>0.\\
 \end{array}\right.\label{1.ssdjddssddffsssferffjkkkkkoooerrfff1}
\end{equation}
When $D(u)\equiv1$, Tao and  Winkler (\cite{Tao793}) proved that problem \dref{1.ssdjddssddffsssferffjkkkkkoooerrfff1} possesses global bounded smooth
solutions in the spatially {\bf two-dimensional} setting, whereas in the {\bf three}-dimensional
counterpart at least global weak solutions can be constructed which eventually become
smooth and bounded.
When $D(u)\geq C_D(u + 1)^{m-1}$ satisfies \dref{91derfff61}--\dref{9162}
 with $m > \frac{1}{2}$
  in the case $N = 1$ and $m > 2 -\frac{2}{N}$ in the case $N\geq 2$,
  it is shown that system \dref{1.ssdjddssddffsssferffjkkkkkoooerrfff1} admits  a unique global classical solution that is uniformly
bounded (\cite{Wang79k1}), while  if $m>2-\frac{6}{N+4}$ ($N\geq3$), \dref{1.ssdjddssddffsssferffjkkkkkoooerrfff1} has a unique global classical solution
(see Zheng and Wang \cite{Zhengssssssdefr23}), which improves the results of \cite{Wangdd79k1}. Apart from the aforementioned system, a source of logistic type is included in \dref{1.ssdjddssddffsssferffjkkkkkoooerrfff1} to describe
the spontaneous growth of cells. The effect of preventing ultimate growth has been widely studied
\cite{Jinfftggg,Lankeit566672,Zheng4556677770}. For instance,
in {\bf three} dimensional case and $D(u)\equiv1$, Zheng and  Mu (\cite{Zheng4556677770}) proved  that the system \dref{1.ssderrfff1} admits a unique global classical solution if the initial datum of $v$ is small; while if $\mu$ is appropriately {\bf large},
Lankeit and Wang (\cite{Lankeit566672}) obtained the global boundedness classical solutions of \dref{1.ssderrfff1} for any {\bf large} initial data, and for any $\mu>0$, they also established the existence of global weak solutions. Recently, if $N=3$, Jin (\cite{Jinfftggg})  showed  that for any $m >1, \mu >0$ and for any {\bf large} initial datum, the problem \dref{1.ssderrfff1} admits a global bounded solution.
Note that the global existence and boundedness of solutions to \dref{1.ssderrfff1} is still open  in {\bf higher dimensions} ($N>4$).
It is the purpose of the present paper to clarify the issue of boundedness to solutions of \dref{1.ssderrfff1} without any restriction on the {\bf space} dimension. Our main result is the following:
\begin{theorem}\label{theorem3}Assume
that $u_0\in C^0(\bar{\Omega})$ and $v_0\in W^{1,\infty}(\bar{\Omega})$ both are nonnegative,
$D$ satisfies \dref{91derfff61}--\dref{9162}.
  If 
$$
 m>\left\{\begin{array}{ll}
1-\frac{\mu}{\chi[1+\lambda_{0}\|v_0\|_{L^\infty(\Omega)}2^{3}]}~~\mbox{if}~~
N\leq2,\\
 1~~~~~~\mbox{if}~~ N\geq3,\\
 \end{array}\right.
$$
%
%
%
%
then there exists a pair $(u,v)\in (C^0(\bar{\Omega}\times[0,\infty))\cap C^{2,1}(\Omega\times(0,\infty))^2$ which solves \dref{1.ssderrfff1} in the classical sense.
Moreover, both $u$ and $v$ are bounded in $\Omega\times(0,\infty)$.
\end{theorem}
\begin{remark}
(i) If $N=3$ and $m=1$, then Theorem \ref{theorem3}
is consistent with the result of Jin (\cite{Jinfftggg}).



%


 (ii)
 If $N=3$, 
 Theorem \ref{theorem3}  extends the results of   Winkler (\cite{Winkler37103}), who proved the possibility of boundness, in the cases
 $\mu>0$ is {\bf sufficiently large}, and with $\Omega\subset \mathbb{R}^N$ is a
convex bounded domains.

\end{remark}

If $D(n) \equiv1$  and $\mu=0$, $-uv$ in the $v$-equation is replaced by $-v+u$, then \dref{1.ssderrfff1} becomes 
%
the well-known Keller-Segel model introduced by Keller and Segel (see Keller and Segel \cite{Keller79,Keller2710}) in 1970:
\begin{equation}
 \left\{\begin{array}{ll}
  u_t=\Delta u-\chi\nabla\cdot(u\nabla v),\quad
x\in \Omega, t>0,\\
 \disp{ v_t=\Delta v - v+u},\quad
x\in \Omega, t>0.\\
 \end{array}\right.\label{1.ssderrfddffyyuff1}
\end{equation}
Over the last decades, the Keller--Segel model has been extensively investigated; in particular, a large amount of work has been devoted to determining whether the solutions are global in time or blow up in finite time, see, for example, 
Cie\'{s}lak et al.
 \cite{Cie72,Cie794}, Burger et al. \cite{Burger2710},
   Horstmann and  Winkler \cite{Horstmann791,Winkler793}
  and references therein.
Additionally, recent studies have shown that the solution behavior can be also impacted by the volume-filling or prevention of overcrowding (see Calvez and Carrillo \cite{Calvez710}, Hillen and Painter \cite{Hillen79}), the nonlinear diffusion (see Zheng \cite{Zheng0,Zheng33312186,Zhddengssdeeezseeddd0}, Ishida et al. \cite{Ishida}, Tao and Winkler \cite{Tao794,Winkler72}]), and the logistic damping (see  Wang et al. \cite{Wang79,Wang76}, Winkler and  Tello \cite{Tello710,Winkler37103},  Zheng and Wang \cite{Zhenssdssdddfffgghjjkk1,Zhengssss6677788888ssdefr23,Zhengsssddssddsseedssddxxss}).

\section{Preliminaries}
%
In order to prove the main results, we first state several elementary
lemmas which will be needed later.
\begin{lemma}(\cite{Hajaiej,Ishida,Zhenssdssdddfffgghjjkk1,Zhddengssdeeezseeddd0})\label{lemma41ffgg}
Let  $s\geq1$ and $q\geq1$.
Assume that $p >0$ and $a\in(0,1)$ satisfy
$$\frac{1}{2}-\frac{p}{N}=(1-a)\frac{q}{s}+a(\frac{1}{2}-\frac{1}{N})~~\mbox{and}~~p\leq a.$$
Then there exist $c_0, c'_0 >0$ such that for all $u\in W^{1,2}(\Omega)\cap L^{\frac{s}{q}}(\Omega)$,
$$\|u\|_{W^{p,2}(\Omega)} \leq c_{0}\|\nabla u\|_{L^2(\Omega)}^{a}\|u\|^{1-a}_{L^{\frac{s}{q}}(\Omega)}+c'_0\|u\|_{L^{\frac{s}{q}}(\Omega)}.$$
\end{lemma}

\begin{lemma}(\cite{Tao41215,Zhengsssddssddsseedssddxxss})\label{lemma630}
Let $T>0$, $\tau\in(0,T)$, $A>0$ and $B>0$, and suppose that
$y:[0,T)\rightarrow[0,\infty)$ is absolutely
continuous such that
\begin{equation}\label{x1.73142hjkl}
\begin{array}{ll}
\displaystyle{
 y'(t)+Ay(t)\leq h(t)}~~\mbox{for a.e.}~~t\in(0,T)\\
\end{array}
\end{equation}
with some nonnegative function $h\in  L^1_{loc}([0, T))$ satisfying
$$\int_{t}^{t+\tau}h(s)ds\leq B~~\mbox{for all}~~t\in(0,T-\tau).$$
Then
%
%
$$y(t)\leq\max\left\{y_0+B,\frac{B}{A\tau}+2B\right\}~~\mbox{for all}~~t\in(0,T).$$
\end{lemma}

\begin{lemma}\label{lemma45xy1222232} (\cite{Zhengwwwwssddghjjkk1})
Suppose  $\gamma\in (1,+\infty)$, $g\in L^\gamma((0, T); L^\gamma(
\Omega))$. 
Let $v$ be a  solution of the following initial boundary value
 \begin{equation}
 \left\{\begin{array}{ll}
v_t -\Delta v+v=f,~~~(x, t)\in
 \Omega\times(0, T ),\\
\disp\frac{\partial v}{\partial \nu}=0,~~~(x, t)\in
 \partial\Omega\times(0, T ),\\
v(x,0)=v_0(x),~~~(x, t)\in
 \Omega.\\
 \end{array}\right.\label{1.3xcx29}
\end{equation}
Then there exists a positive constant $\lambda_0:=\lambda_0(\Omega,\gamma,N)$ such that if $s_0\in[0,T)$, $v(\cdot,s_0)\in W^{2,\gamma}(\Omega)(\gamma>N)$ with $\disp\frac{\partial v(\cdot,s_0)}{\partial \nu}=0,$ then
\begin{equation}
\begin{array}{rl}
&\disp{\int_{s_0}^Te^{\gamma s}\| v(\cdot,t)\|^{\gamma}_{W^{2,\gamma}(\Omega)}ds\leq\lambda_0\left(\int_{s_0}^Te^{\gamma s}
\|g(\cdot,s)\|^{\gamma}_{L^{\gamma}(\Omega)}ds+e^{\gamma s_0}(\|v_0(\cdot,s_0)\|^{\gamma}_{W^{2,\gamma}(\Omega)})\right).}\\
\end{array}
\label{cz2.5bbv114}
\end{equation}
\end{lemma}

\begin{lemma}\label{drfe116lemma70hhjj}(Lemma 2.2 of \cite{Zhengssssssdefr23})
Suppose that $\beta>\max\{1,\frac{N-2}{2}\}$ and $\Omega\subset \mathbb{R}^N(N\geq2)$ is a bounded domain with smooth boundary.
Moreover, assume that
 \begin{equation}\lambda\in[2\beta+2,L_{\beta,N}],
\label{3.10deerfgghhjuuloollgghhhyhh}
\end{equation}
where
\begin{equation*}
L_{\beta,N}\left\{\begin{array}{ll}
=\frac{N(2\beta+1)-2(\beta+1)}{N-2}~~\mbox{if}~~
N\geq3,\\
 <+\infty~~\mbox{if}~~N=2.
 \end{array}\right.
\end{equation*}
%
%
Then there exists $C > 0$ such that for all $\varphi\in C^2(\bar{\Omega})$ fulfilling $\varphi\cdot\frac{\partial\varphi}{\partial\nu}= 0$
 on $\partial\Omega$
 we have
 \begin{equation}
 \begin{array}{rl}
 &\|\nabla\varphi\|_{L^\lambda(\Omega)}\leq C\||\nabla\varphi|^{\beta-1}D^2\varphi\|_{L^2(\Omega)}^{\frac{2(\lambda-N)}{(2\beta-N+2)\lambda}}
 \|\varphi\|_{L^\infty(\Omega)}^{\frac{2N\beta-(N-2)\lambda}{(2\beta-N+2)\lambda}}+C\|\varphi\|_{L^\infty(\Omega)}.\\
\end{array}\label{aqwswddaassffssff3.10deerfgghhjuuloollgghhhyhh}
\end{equation}
\end{lemma}

The following local existence result is rather standard, since a similar reasoning in
 \cite{Cie72,Wang79,Wang76,Wang72,Zheng0}.
 We omit it here.
\begin{lemma}\label{lemma70}
Suppose that $\Omega \subset \mathbb{R}^N (N\geq 1)$
 is a bounded domain with smooth boundary, $D$
satisfies \dref{91derfff61}--\dref{9162}.
Then for
 nonnegative triple
$(u_0(x), v_0(x))\in
C(\bar{\Omega})\times W^{1,\infty}(\bar{\Omega})$,
Then problem (1.3)
has a unique local-in-time non-negative classical
functions
\begin{equation}
 \left\{\begin{array}{ll}
   u\in  C^0(\bar{\Omega}\times[0,T_{max}))\cap C^{2,1}(\Omega\times(0,T_{max})),\\
    v \in   C^0(\bar{\Omega}\times[0,T_{max}))\cap C^{2,1}(\Omega\times(0,T_{max})),\\
 \end{array}\right.\label{dffff1cvvfgg.1fghyuisdakkklll}
\end{equation}
where $T_{max}$ denotes the maximal existence time. 
%
%
%
Moreover, if  $T_{max}<+\infty$, then
\begin{equation}
\|u(\cdot, t)\|_{L^\infty(\Omega)}\rightarrow\infty~~ \mbox{as}~~ t\nearrow T_{max}
\label{1.163072x}
\end{equation}
is fulfilled.
\end{lemma}

\begin{lemma}\label{ssdeedrfe116lemma70hhjj} (Lemma 3.2 of \cite{Jinfftggg})
There exists 
$ C > 0$ such that
 the solution $(u, v)$ of  \dref{1.ssderrfff1} satisfies
%
%
\begin{equation}
 \begin{array}{rl}
 \|u(\cdot,t)\|_{L^1(\Omega)}=\|u_0\|_{L^1(\Omega)}~~~\mbox{for all}~~t\in (0, T_{max}),
\end{array}\label{ssddaqwswddaassffssff3.ddfvbb10deerfgghhjuuloollgghhhyhh}
\end{equation}
  \begin{equation}
 \begin{array}{rl}
 \|v(\cdot,t)\|_{L^\infty(\Omega)}\leq\|v_0\|_{L^\infty(\Omega)}~~~\mbox{for all}~~t\in (0, T_{max})
\end{array}\label{ssddaqwswddaassffssff3.10deerfgghhjuuloollgghhhyhh}
\end{equation}
and
\begin{equation}
\int_t^{t+\tau}\int_{\Omega}{u^{2}}\leq  C~~\mbox{for all}~~ t\in(0, T_{max}-\tau),
\label{bnmbncz2.5ghhjuyuivvddfggghhbssdddeennihjj}
\end{equation}
where \begin{equation}
\tau:=\min\{1,\frac{1}{6}T_{max}\}.
\label{cz2.5ghju48cfg924vbhu}
\end{equation}
%
\end{lemma}
Now, collecting Lemma \ref{drfe116lemma70hhjj} and Lemma \ref{ssdeedrfe116lemma70hhjj}, we derive that:
\begin{lemma}\label{ghyuushhuuusdeedrfe116lemma70hhjj}
Let $N\geq3$ and $\beta>\max\{1,\frac{N-2}{2}\}.$ Then 
there exists a positive constant 
$\kappa_0$
such that the solution of \dref{1.ssderrfff1} satisfies
  \begin{equation}
 \begin{array}{rl}
 &\|\nabla v\|_{L^{2\beta+2}(\Omega)}^{2\beta+2}\leq \kappa_0(\||\nabla v|^{\beta-1}D^2 v\|_{L^2(\Omega)}^{2}
 +1).\\
\end{array}\label{aqwswdggyuuudaassffssff3.10deerfgghhjuuloollgghhhyhh}
\end{equation}
\end{lemma}
\begin{proof}
Let $\varphi=v$ and $\lambda=2\beta+2$ in Lemma \ref{drfe116lemma70hhjj}, then by using  $\frac{2(\lambda-N)}{(2\beta-N+2)\lambda}\lambda=2$ and \dref{ssddaqwswddaassffssff3.10deerfgghhjuuloollgghhhyhh}, we can obtain the result.
\end{proof}
\section{A priori estimates}

In this section, we are going to establish an iteration step to develop the main ingredient of our result.
Firstly, employing almost exactly the same arguments as in the proof of Lemma 2.1 in \cite{Winkler37103} (see also Lemma 3.2 of \cite{Jinfftggg}), we may derive the following Lemma:
%
%
%
\begin{lemma}\label{wsdelemma45}
Under the assumptions in theorem \ref{theorem3}, we derive that
there exists a positive constant 
$C$ 
such that the solution of \dref{1.ssderrfff1} satisfies
%
%
\begin{equation}
\int_{\Omega}|\nabla {v}(x,t)|^2 \leq C~~\mbox{for all}~~ t\in(0, T_{max})
\label{cz2.5ghju48cfg924ghyuji}
\end{equation}
and
\begin{equation}
\begin{array}{rl}
&\disp{\int_{t}^{t+\tau}\int_{\Omega}[|\nabla {v}|^2+u^2+ |\Delta {v}|^2]\leq C~~\mbox{for all}~~ t\in(0, T_{max}-\tau),}\\
\end{array}
\label{bnmbncz2.5ghhjuyuivvbssdddeennihjj}
\end{equation}
where $\tau$ is given by \dref{cz2.5ghju48cfg924vbhu}.
\end{lemma}

\begin{lemma}\label{lemmaghjssddgghhmk4563025xxhjklojjkkk}
Under the conditions of Theorem \ref{theorem3},
there exists $C>0$  
 such that the solution of \dref{1.ssderrfff1} satisfies
\begin{equation}
\begin{array}{rl}
&\disp{\int_{\Omega}u\ln u\leq C}\\
\end{array}
\label{czfvgb2.5ghhjuyuiissswwerrtthjj}
\end{equation}
for all $t\in(0, T_{max})$.
Moreover,  for each $T\in(0, T_{max})$,
 one can find a constant $C > 0$ 
  such that
\begin{equation}
\begin{array}{rl}
&\disp{\int_{0}^T\int_{\Omega}  \frac{D(u)|\nabla u|^2}{u}\leq C(T+1)}\\
\end{array}
\label{bnmbncz2.5ghhjuyuiihjj}
\end{equation}
%
as well as
\begin{equation}
\begin{array}{rl}
&\disp{\int_{0}^T\int_{\Omega}   u^2(\ln u+1)\leq C(T+1).}\\
\end{array}
\label{cvffvbssdgvvcz2.ffghhjj5ghhjuyuiihjj}
\end{equation}

\end{lemma}
\begin{proof}
First, testing the first equation in \dref{1.ssderrfff1} by $\ln u$ yields
\begin{equation}
\begin{array}{rl}
&\disp{\frac{d}{dt}\int_{\Omega}u\ln u}\\
=&\disp{\int_{\Omega}u_{ t}\ln u+u_{ t}}\\
=&\disp{-\int_{\Omega}\frac{D(u)|\nabla u|^2}{u}+\chi\int_{\Omega}  \nabla u\cdot\nabla v-\mu\int_{\Omega}u^2\ln u}\\
&\disp{+\mu\int_{\Omega}u\ln u-\mu\int_{\Omega}u^2+\mu\int_{\Omega}u~~~\mbox{for all}~~ t\in (0, T_{max}).}\\
\end{array}
\label{czfvgb2.5ghhjuyuddffccvviihjj}
\end{equation}
On the other hand, with some basic calculation and using the Young inequality and \dref{cz2.5ghju48cfg924ghyuji}, one can get 
\begin{equation}
\begin{array}{rl}
&\disp{-\mu\int_{\Omega}u^2-\mu\int_{\Omega}u^2\ln u+(\mu+1)\int_{\Omega}u\ln u}\\
\leq&\disp{-\frac{\mu}{2}\int_{\Omega}u^2\ln (u+1)-\frac{\mu}{2}\int_{\Omega}u^2+C_1
~~~\mbox{for all}~~ t\in (0, T_{max})}\\
\end{array}
\label{czfvgb2.5ghhjuyudssddeedffccvviihjj}
\end{equation}
with some positive constant  $C_1.$
Here we have use the fact that
$y\ln y\geq-\frac{1}{e}$ for any $y>0.$
Next, 
once more integrating by parts and using the Young inequality, we derive
\begin{equation}
\begin{array}{rl}
\chi\disp\int_{\Omega}  \nabla u\cdot\nabla v
=&\disp{-\chi\int_{\Omega}    u \Delta v}\\
\leq&\disp{\frac{\mu}{4}\int_{\Omega}   u^2+\frac{\chi^2}{\mu}\int_{\Omega}   |\Delta v|^2~~~\mbox{for all}~~ t\in (0, T_{max}).}\\
\end{array}
\label{czfvgb2.5ghhjssdeeaauyudssddeedffccvviihjj}
\end{equation}

Putting the estimates \dref{czfvgb2.5ghhjuyudssddeedffccvviihjj} and  \dref{czfvgb2.5ghhjssdeeaauyudssddeedffccvviihjj} into \dref{czfvgb2.5ghhjuyuddffccvviihjj} and using \dref{cz2.5ghju48cfg924ghyuji}, then there exists a positive constant $C_{2}$ such that
\begin{equation}
\begin{array}{rl}
&\disp{\frac{d}{dt}\int_{\Omega}u\ln u+\int_{\Omega}u\ln u+\int_{\Omega}\frac{D(u)|\nabla u|^2}{u}+\frac{\mu}{4}\int_{\Omega}u^2\ln (u+1)+
\frac{\mu}{4}\int_{\Omega}u^2}\\
\leq&\disp{\frac{\chi^2}{\mu}\int_{\Omega} |\Delta v|^2+C_{2}~~~\mbox{for all}~~ t\in (0, T_{max}),}\\
\end{array}
\label{czfvgb2.5ghhjussderfrfyuddffccvviihjj}
\end{equation}
which, together with Lemma \ref{lemma630} and \dref{bnmbncz2.5ghhjuyuivvbssdddeennihjj}, gives
\dref{czfvgb2.5ghhjuyuiissswwerrtthjj}--\dref{cvffvbssdgvvcz2.ffghhjj5ghhjuyuiihjj}. The proof of Lemma \ref{lemmaghjssddgghhmk4563025xxhjklojjkkk} is completed.
\end{proof}

\begin{lemma}\label{lemmadderr45630223}
Assume that    $m>1-\frac{\mu}{\chi[1+\lambda_{0}\|v_0\|_{L^\infty(\Omega)}2^{3}]}$ and $N\leq2$.
    Let $(u,v)$ be a solution to \dref{1.ssderrfff1} on $(0,T_{max})$.
 Then 
 for all $p>1$,
there exists a positive constant $C:=C(p,|\Omega|,\mu,\chi)$ such that 
\begin{equation}
\int_{\Omega}u^p(x,t) \leq C ~~~\mbox{for all}~~ t\in(0,T_{max}).
\label{zjscz2.ddffrr5297x96302222114}
\end{equation}
\end{lemma}
\begin{proof}
Firstly,  let us  pick any $s_0\in(0,T_{max})$ and $s_0\leq1$. Then 
 by Lemma \ref{lemma70}, we can conclude that for any given $s_0\in(0,T_{max}),s_0\leq1,$ there exists
$K>0$ such that
\begin{equation}\label{eqx45xx12112}
\|u(\tau)\|_{L^\infty(\Omega)}\leq K,~~\|v(\tau)\|_{L^\infty(\Omega)}\leq K~~\mbox{and}~~\|\Delta v(\tau)\|_{L^\infty(\Omega)}\leq K~~\mbox{for all}~~\tau\in[0,s_0].
\end{equation}
Assume that  $1<p<2$.
Multiplying the first equation of \dref{1.ssderrfff1}
  by $u^{{p}-1}$, integrating over $\Omega$ and using \dref{9162},
 we get
\begin{equation}
\begin{array}{rl}
&\disp{\frac{1}{{p}}\frac{d}{dt}\|u\|^{{p}}_{L^{{p}}(\Omega)}+({{p}-1})\int_{\Omega}u^{m+{{p}-3}}|\nabla u|^2}
\\
\leq&\disp{-\chi\int_\Omega \nabla\cdot( u\nabla v)
  u^{{p}-1} +
\int_\Omega   u^{{p}-1}(\mu u-\mu u^2)~~\mbox{for all}~~t\in (0, T_{max}),}\\
\end{array}
\label{99922cz2.5114114}
\end{equation}
which implies that,
\begin{equation}
\begin{array}{rl}
\disp\frac{1}{{p}}\disp\frac{d}{dt}\|u\|^{{{p}}}_{L^{{p}}(\Omega)}\leq&\disp{-\frac{{p}+1}{{p}}\int_{\Omega} u^{p} -\chi\int_\Omega \nabla\cdot( u\nabla v)
  u^{{p}-1} }\\
 &+\disp{\int_\Omega \left(\frac{{p}+1}{{p}} u^{p}+  u^{{p}-1}(\mu u-\mu u^2)\right)~~\mbox{for all}~~t\in (0, T_{max}).}\\
\end{array}
\label{cz2.5kk1214114114}
\end{equation}
Next, we derive from the  Young inequality  that
\begin{equation}
\begin{array}{rl}
&\disp{\int_\Omega  \left(\frac{{p}+1}{{p}} u^{p}+ u^{{p}-1}(\mu u-\mu u^2)\right)}\\
= &\disp{\frac{{p}+1}{{p}}\int_\Omega u^{p} +\mu\int_\Omega    u^{{p}}- \mu\int_\Omega  u^{{{p}+1}}}\\
\leq &\disp{(\varepsilon_1- \mu)\int_\Omega u^{{{p}+1}} +C_1(\varepsilon_1,{p})~~\mbox{for all}~~t\in (0, T_{max}),}
\end{array}
\label{cz2.563011228ddff}
\end{equation}
where $$C_1(\varepsilon_1,{p})=\frac{1}{{p}+1}\left(\varepsilon_1\frac{{p}+1}{{p}}\right)^{-{p} }
\left(\frac{{p}+1}{{p}}+\mu\right)^{{p}+1 }|\Omega|.$$
%
%
Now,
integrating by parts to the first term on the right hand side of \dref{99922cz2.5114114} and  using  the Young inequality,
we conclude that 
\begin{equation}
\begin{array}{rl}
&\disp{-\chi\int_\Omega \nabla\cdot( u\nabla v)
  u^{{p}-1} }
\\
=&\disp{({{p}-1})\chi\int_\Omega  u^{{{p}-1}}\nabla u\cdot\nabla v }
\\
\leq&\disp{\frac{{{p}-1}}{{p}}\chi \int_\Omega u^{{p}}|\Delta v|}
\\
\leq&\disp{(p-1)\chi \int_\Omega u^{{p}}|\Delta v|}
\\
=&\disp{(p-1)^{\frac{1}{p+1}+\frac{p}{p+1}}\chi \int_\Omega u^{{p}}|\Delta v|}
\\
\leq&\disp{(p-1)\chi\int_\Omega u^{{p+1}}+(p-1)\chi \int_\Omega |\Delta v|^{p+1}~~\mbox{for all}~~t\in (0, T_{max}).}
\\
\end{array}
\label{66788cz2.563019114}
\end{equation}
Thus, inserting \dref{66788cz2.563019114} into \dref{cz2.5kk1214114114}, we conclude that
\begin{equation*}
\begin{array}{rl}
\disp\frac{1}{{p}}\disp\frac{d}{dt}\|u\|^{{{p}}}_{L^{{p}}(\Omega)}\leq&\disp{(\varepsilon_1+(p-1)\chi- \mu)\int_\Omega u^{{{p}+1}} dx-\frac{{p}+1}{{p}}\int_{\Omega} u^{p} dx}\\
&+\disp{(p-1)\chi\int_\Omega |\Delta v|^{ {p}+1} dx+
C_1(\varepsilon_1,{p}).}\\
\end{array}
\end{equation*}
For any $t\in (s_0,T_{max})$,
employing the variation-of-constants formula to the above inequality, we obtain
\begin{equation}
\begin{array}{rl}
&\disp{\frac{1}{{p}}\|u(t) \|^{{{p}}}_{L^{{p}}(\Omega)}}
\\
\leq&\disp{\frac{1}{{p}}e^{-( {p}+1)(t-s_0)}\|u(s_0) \|^{{{p}}}_{L^{{p}}(\Omega)}+(\varepsilon_1+(p-1)\chi - \mu)\int_{s_0}^t
e^{-( {p}+1)(t-s)}\int_\Omega u^{{{p}+1}} dxds}\\
&+\disp{(p-1)\chi\int_{s_0}^t
e^{-( {p}+1)(t-s)}\int_\Omega |\Delta v|^{ {p}+1} dxds+ C_1(\varepsilon_1,{p})\int_{s_0}^t
e^{-( {p}+1)(t-s)}ds}\\
\leq&\disp{(\varepsilon_1+(p-1)\chi - \mu)\int_{s_0}^t
e^{-( {p}+1)(t-s)}\int_\Omega u^{{{p}+1}} dxds}\\
&+\disp{(p-1)\chi \int_{s_0}^t
e^{-( {p}+1)(t-s)}\int_\Omega |\Delta v|^{ {p}+1} dxds+C_2({p},\varepsilon_1),}\\
\end{array}
\label{cz2.5kk1214114114rrgg}
\end{equation}
where
$$C_2:=C_2({p},\varepsilon_1):=\frac{1}{{p}}\|u(s_0) \|^{{{p}}}_{L^{{p}}(\Omega)}+
 C_1(\varepsilon_1,{p})\int_{s_0}^t
e^{-( {p}+1)(t-s)}ds.$$
Let $ t \in (s_0, T_{max})$ and rewrite the second  equation as
$$v_t -\Delta v+ v=-vu+v.$$
%
%
%
Now, by Lemma \ref{lemma45xy1222232} and \dref{ssddaqwswddaassffssff3.10deerfgghhjuuloollgghhhyhh}, we have
\begin{equation}\label{cz2.5kk1214114114rrggjjkk}
\begin{array}{rl}
&\disp{(p-1)\chi \int_{s_0}^t
e^{-( {p}+1)(t-s)}\int_\Omega |\Delta v|^{ {p}+1} dxds}
\\
=&\disp{(p-1)\chi e^{-( {p}+1)t}\int_{s_0}^t
e^{( {p}+1)s}\int_\Omega |\Delta v|^{ {p}+1} dxds}\\
\leq&\disp{(p-1)\chi e^{-( {p}+1)t}\lambda_0[\int_{s_0}^t
\int_\Omega e^{( {p}+1)s}|-vu+v|^{ {p}+1} dxds}\\
&+\disp{e^{( {p}+1)s_0}\|v(s_0,t)\|^{ {p}+1}_{W^{2, {p}+1}}]}\\
\leq&\disp{(p-1)\chi e^{-( {p}+1)t}\lambda_0[\|v_0\|_{L^\infty(\Omega)}2^{p+1}\int_{s_0}^t
\int_\Omega e^{( {p}+1)s}(u^{ {p}+1}+1) dxds}\\
&+\disp{e^{( {p}+1)s_0}\|v(s_0,t)\|^{ {p}+1}_{W^{2, {p}+1}}]~~\mbox{for all}~~t\in (0, T_{max}),}\\
\end{array}
\end{equation}
where $\lambda_0$ is the same as Lemma \ref{lemma45xy1222232}.
By substituting \dref{cz2.5kk1214114114rrggjjkk} into \dref{cz2.5kk1214114114rrgg}, 
we get
\begin{equation}
\begin{array}{rl}
&\disp{\frac{1}{{p}}\|u(t) \|^{{{p}}}_{L^{{p}}(\Omega)}}
\\
\leq&\disp{(\varepsilon_1+(p-1)\chi +(p-1)\chi \lambda_0\|v_0\|_{L^\infty(\Omega)}2^{p+1}- \mu)\int_{s_0}^t
e^{-( {p}+1)(t-s)}\int_\Omega u^{{{p}+1}} dxds}\\
&+\disp{(p-1)\chi e^{-( {p}+1)(t-s_0)}\lambda_0\|v(s_0,t)\|^{ {p}+1}_{W^{2, {p}+1}}}\\
&+\disp{(p-1)\chi e^{-( {p}+1)t}\lambda_0\|v_0\|_{L^\infty(\Omega)}2^{p+1}\int_{s_0}^t
\int_\Omega e^{( {p}+1)s} dxds+C_2({p},
\varepsilon_1).}\\
\end{array}
\label{cz2.5kk1214114114rrggkkll}
\end{equation}
Now, choosing $p:={p_0}:=1+\frac{\mu}{\chi[1+\lambda_{0}\|v_0\|_{L^\infty(\Omega)}2^{3}]}>1$  
in \dref{cz2.5kk1214114114rrggkkll}  and using $p<2$, then we conclude that
$$
\begin{array}{rl}
\mu=&\disp{(p_0-1)\chi +(p_0-1)\chi \lambda_0\|v_0\|_{L^\infty(\Omega)}2^{3}}\\
>&\disp{(p_0-1)\chi +(p_0-1)\chi \lambda_0\|v_0\|_{L^\infty(\Omega)}2^{p_0+1}.}\\
\end{array}
$$
Thus, picking  $\varepsilon_1$ appropriately   small such that
$$0<\varepsilon_1<\mu -(p_0-1)\chi +(p_0-1)\chi \lambda_0\|v_0\|_{L^\infty(\Omega)}2^{p_0+1},$$
then in light of \dref{cz2.5kk1214114114rrggkkll}, we derive that there exists a positive constant $C_3$
such that
\begin{equation}
\begin{array}{rl}
&\disp{\int_{\Omega}u^{{p_0}}(x,t) dx\leq C_3~~\mbox{for all}~~t\in (s_0, T_{max}).}\\
\end{array}
\label{cz2.5kk1214114114rrggkklljjuu}
\end{equation}
Next,
we fix $p <\frac{2{p_0}}{(2-{p_0})^+}$
and choose some
 $\alpha> \frac{1}{2}$ such that
\begin{equation}
p <\frac{1}{\frac{1}{p_0}-\frac{1}{2}+\frac{2}{2}(\alpha-\frac{1}{2})}\leq\frac{2{p_0}}{(2-{p_0})^+}.
\label{fghgbhnjcz2.5ghju48cfg924ghyuji}
\end{equation}

Now, involving the variation-of-constants formula
for $v$, we have
\begin{equation}
v(t)=e^{-\tau(A+1)}v(s_0) +\int_{s_0}^{t}e^{-(t-s)(A+1)}(-v(s)u(s)+v(s)) ds,~~ t\in(s_0, T_{max}).
\label{fghbnmcz2.5ghju48cfg924ghyuji}
\end{equation}
Hence, it follows from \dref{eqx45xx12112}, \dref{ssddaqwswddaassffssff3.10deerfgghhjuuloollgghhhyhh} and  
 \dref{fghbnmcz2.5ghju48cfg924ghyuji} that
\begin{equation}
\begin{array}{rl}
&\disp{\|(A+1)^\alpha v(t)\|_{L^p(\Omega)}}\\
\leq&\disp{C_4\int_{s_0}^{t}(t-s)^{-\alpha-\frac{2}{2}(\frac{1}{p_0}-\frac{1}{p})}e^{-\mu(t-s)}\|-v(s)u(s)+v(s)\|_{L^{p_0}(\Omega)}ds+
C_4s_0^{-\alpha-\frac{2}{2}(1-\frac{1}{p})}\|v(s_0,t)\|_{L^1(\Omega)}}\\
\leq&\disp{C_5\int_{0}^{+\infty}\sigma^{-\alpha-\frac{2}{2}(\frac{1}{p_0}-\frac{1}{p})}e^{-\mu\sigma}d\sigma
+C_5s_0^{-\alpha-\frac{2}{2}(1-\frac{1}{p})}K.}\\
\end{array}
\label{gnhmkfghbnmcz2.5ghju48cfg924ghyuji}
\end{equation}
Hence, due to \dref{fghgbhnjcz2.5ghju48cfg924ghyuji}  and \dref{gnhmkfghbnmcz2.5ghju48cfg924ghyuji}, we have
\begin{equation}
\int_{\Omega}|\nabla {v}(t)|^{p}\leq C_6~~\mbox{for all}~~ t\in(s_0, T_{max})~~\mbox{and}~~p\in[1,\frac{2{p_0}}{(2-{p_0})^+}).
\label{ffgbbcz2.5ghju48cfg924ghyuji}
\end{equation}
Finally, in view of \dref{eqx45xx12112} and \dref{ffgbbcz2.5ghju48cfg924ghyuji},
 we can get \begin{equation}
\int_{\Omega}|\nabla {v}(t)|^{p}\leq C_7~~\mbox{for all}~~ t\in(0, T_{max})~~\mbox{and}~~p\in[1,\frac{2{p_0}}{(2-{p_0})^+})
\label{ffgbbcz2.5ghjusseeeddd48cfg924ghyuji}
\end{equation}
with some positive constant $C_7.$

Next, for any $p>1,$
multiplying both sides of the first equation in \dref{1.ssderrfff1} by $u^{p-1}$, integrating over $\Omega$, integrating by parts and using \dref{9162}, we arrive at
\begin{equation}
\begin{array}{rl}
&\disp{\frac{1}{{p}}\frac{d}{dt}\|u\|^{{p}}_{L^{{p}}(\Omega)}+({{p}-1})\int_{\Omega}u^{m+{{p}-3}}|\nabla u|^2dx}
\\
\leq&\disp{-\chi\int_\Omega \nabla\cdot( u\nabla v)
  u^{{p}-1} dx+
\int_\Omega   u^{{p}-1}(\mu u-\mu u^2) dx}\\
=&\disp{\chi({p}-1)\int_\Omega  u^{{p}-1}\nabla u\cdot\nabla v
   dx+
\int_\Omega   u^{{p}-1}(\mu u-\mu u^2) dx,}\\
\end{array}
\label{cz2aasweee.5114114}
\end{equation}
which together with the Young inequality implies that
\begin{equation}
\begin{array}{rl}
&\disp{\frac{1}{{p}}\frac{d}{dt}\|u\|^{{p}}_{L^{{p}}(\Omega)}+({{p}-1})\int_{\Omega}u^{m+{{p}-3}}|\nabla u|^2dx}
\\
\leq&\disp{\frac{{{p}-1}}{2}\int_{\Omega}u^{m+{{p}-3}}|\nabla u|^2dx+
\frac{\chi^2({p}-1)}{2}\int_\Omega  u^{{p}+1-m}|\nabla v|^2
   dx-\frac{\mu }{2}\int_\Omega u^{p+1}dx+ C_8}\\
\end{array}
\label{888888cz2aasweee.5ssedfff114114}
\end{equation}
for some positive constant $C_8.$
%
We choose  $1<q_0<\frac{2{p_0}}{2(2-{p_0})^+}$ which is close to $\frac{2{p_0}}{2(2-{p_0})^+}$. In light of the H\"{o}lder inequality and \dref{ffgbbcz2.5ghjusseeeddd48cfg924ghyuji}, we derive   at
\begin{equation}
\begin{array}{rl}
 \disp\frac{\chi^2({p}-1)}{2}\disp\int_\Omega{{u^{p+1-m}}} |\nabla {v}|^2\leq&\disp{ \disp\frac{\chi^2({p}-1)}{2}\left(\disp\int_\Omega{{u^{\frac{q_0}{q_0-1} (p+1-m) }}}\right)^{\frac{q_0-1}{q_0}}\left(\disp\int_\Omega |\nabla {v}|^{2q_0}\right)^{\frac{1}{q_0}}}\\
\leq&\disp{C_9\|  {{u^{\frac{m+p-1}{2}}}}\|^{2\frac{p+1-m}{m+p-1}}_{L^{2\frac{q_0}{q_0-1}\frac{p+1-m}{m+p-1} }(\Omega)},}\\
\end{array}
\label{cz2.57151hhkkhhhjukildrfthjjhhhhh}
\end{equation}
where $C_9$ is a positive constant.
 Due to  $q_0>1,p>\max\{1-m, m+p_0-1-\frac{p_0}{q_0}\}$,
we have
$$\frac{p_0}{m+p-1}\leq\frac{q_0}{q_0-1}\frac{p+1-m}{m+p-1}<+\infty,$$
which together with the Gagliardo--Nirenberg inequality
 implies that
\begin{equation}
\begin{array}{rl}
C_9\|  {{u^{\frac{m+p-1}{2}}}}\|
^{2\frac{p+1-m}{m+p-1}}_{L^{2\frac{q_0}{q_0-1}\frac{p+1-m}{m+p-1} }(\Omega)}\leq&\disp{C_{10}(\|\nabla   {{u^{\frac{m+p-1}{2}}}}\|_{L^2(\Omega)}^{\mu_1}\|  {{u^{\frac{m+p-1}{2}}}}\|_{L^\frac{2p_0}{m+p-1}(\Omega)}^{1-\mu_1}+\|  {{u^{\frac{m+p-1}{2}}}}\|_{L^\frac{2p_0}{m+p-1}(\Omega)})^{2\frac{p+1-m}{m+p-1}}}\\
\leq&\disp{C_{11}(\|\nabla   {{u^{\frac{m+p-1}{2}}}}\|_{L^2(\Omega)}^{2\frac{p+1-m}{m+p-1}\mu_1}+1)}\\
=&\disp{C_{11}(\|\nabla   {{u^{\frac{m+p-1}{2}}}}\|_{L^2(\Omega)}^{\frac{2[q_0(p+1-m)-p_0(q_0-1)]}{q_0(m+p-1)}}+1)}\\
\end{array}
\label{cz2.563022222ikopl2sdfg44}
\end{equation}
with some positive constants $C_{10}, C_{11}$ and
$$\mu_1=\frac{\frac{2{(m+p-1)}}{2p_0}-\frac{2(m+p-1)(q_0-1)}{2q_0(p+1-m)}}{1-\frac{2}{2}+\frac{2{(m+p-1)}}{2p_0}}=
{(m+p-1)}\frac{\frac{2}{2p_0}-\frac{2(q_0-1)}{2q_0(p+1-m)}}{1-\frac{2}{2}+\frac{2{(m+p-1)}}{2p_0}}\in(0,1).$$
On the other hand, by  $p_0=1+\frac{\mu}{\chi[1+\lambda_{0}\|v_0\|_{L^\infty(\Omega)}2^{3}]},q_0<\frac{2{p_0}}{2(2-{p_0})^+}$ (close to $\frac{2{p_0}}{2(2-{p_0})^+}$) and
$m> 1-\frac{\mu}{\chi[1+\lambda_{0}\|v_0\|_{L^\infty(\Omega)}2^{3}]} $, we derive that
\begin{equation}\frac{q_0(p+1-m)-p_0(q_0-1)}{q_0(m+p-1)}<1.
\label{cz2.563022222ikoddffffpl2sdfg44}
\end{equation}
Hence, in view of the Young inequality,
we have
\begin{equation}
\begin{array}{rl}
\disp\frac{\chi^2({p}-1)}{2}\disp\int_\Omega  u^{{p}+1-m}|\nabla v|^2dx &\leq\disp{\frac{{{p}-1}}{4}\int_{\Omega}u^{m+{{p}-3}}|\nabla u|^2dx+C_{12}.}\\
\end{array}
\label{cz2aasweeeddfff.5ssedfssddff114114}
\end{equation}
Inserting \dref{cz2aasweeeddfff.5ssedfssddff114114} into \dref{888888cz2aasweee.5ssedfff114114}, we conclude that
\begin{equation}
\begin{array}{rl}
&\disp{\frac{1}{{p}}\frac{d}{dt}\|u\|^{{p}}_{L^{{p}}(\Omega)}+\frac{{{p}-1}}{4}\int_{\Omega}u^{m+{{p}-3}}|\nabla u|^2dx+
\frac{\mu}{2}\int_\Omega u^{p+1}dx\leq C_{13}.}\\
\end{array}
\label{cz2aasweee.5ssedfff114114}
\end{equation}
Therefore, integrating the above inequality  with respect to $t$ yields
\begin{equation}
\begin{array}{rl}
\|u(\cdot, t)\|_{L^{{p}}(\Omega)}\leq C_{14} ~~ \mbox{for all}~~p>1~~\mbox{and}~~  t\in(0,T_{max}) \\
\end{array}
\label{cz2.5g556789hhjui78jj90099}
\end{equation}
for some positive constant $C_{14}$.
%
%
The proof Lemma \ref{lemmadderr45630223} is complete.
\end{proof}

\begin{lemma}\label{lemma45630223}
Assume that    $m>1$ and $N\geq3$.
    Let $(u,v)$ be a solution to \dref{1.ssderrfff1} on $(0,T_{max})$.
 Then 
 for all $p>1$,
there exists a positive constant $C:=C(p,|\Omega|,\mu,\chi)$ such that 
\begin{equation}
\int_{\Omega}u^p(x,t) \leq C ~~~\mbox{for all}~~ t\in(0,T_{max}).
\label{zjscz2.5297x96302222114}
\end{equation}
\end{lemma}
\begin{proof}
Let $\beta>\max\{1,\frac{N-2}{2}\}$ and
\begin{equation}\beta<p<\beta+(m-1)(\beta+1).
\label{cz2.5ghju4ssderfyyuuiioo8156}
\end{equation}
Observing that  $\nabla {v}\cdot\nabla\Delta {v}  = \frac{1}{2}\Delta |\nabla {v}|^2-|D^2{v}|^2$, in light of  a straightforward computation using the second
equation in \dref{1.ssderrfff1} and several integrations by parts, we conclude that
\begin{equation}
\begin{array}{rl}
\disp{\frac{1}{{2{\beta} }}\frac{d}{dt} \|\nabla {v}\|^{{{2{\beta} }}}_{L^{{2{\beta} }}(\Omega)}}= &\disp{\int_{\Omega} |\nabla {v}|^{2{\beta} -2}\nabla {v}\cdot\nabla(\Delta {v}-uv)}
\\
=&\disp{\frac{1}{{2}}\int_{\Omega} |\nabla {v}|^{2{\beta} -2}\Delta |\nabla {v}|^2-\int_{\Omega} |\nabla {v}|^{2{\beta} -2}|D^2 {v}|^2}\\
&+\disp{\int_\Omega uv\nabla\cdot( |\nabla {v}|^{2{\beta} -2}\nabla {v})}\\
=&\disp{-\frac{{\beta} -1}{{2}}\int_{\Omega} |\nabla {v}|^{2{\beta} -4}\left|\nabla |\nabla {v}|^{2}\right|^2+\frac{1}{{2}}\int_{\partial\Omega} |\nabla {v}|^{2{\beta} -2}\frac{\partial  |\nabla {v}|^{2}}{\partial\nu}}\\
&-\disp{\int_{\Omega} |\nabla {v}|^{2{\beta} -2}|D^2 {v}|^2+\int_\Omega {u}v
|\nabla {v}|^{2{\beta} -2}\Delta {v}+
\int_\Omega {u}v\nabla {v}\cdot\nabla( |\nabla {v}|^{2{\beta} -2})}\\
=&\disp{-\frac{2({\beta} -1)}{{{\beta} ^2}}\int_{\Omega}\left|\nabla |\nabla {v}|^{{\beta} }\right|^2+\frac{1}{{2}}\int_{\partial\Omega} |\nabla {v}|^{2{\beta} -2}\frac{\partial  |\nabla {v}|^{2}}{\partial\nu}-\int_{\Omega} |\nabla {v}|^{2{\beta} -2}|D^2 {v}|^2}\\
&+\disp{\int_\Omega {u}v |\nabla {v}|^{2{\beta} -2}\Delta {v}+\int_\Omega {u}v\nabla {v}\cdot\nabla( |\nabla {v}|^{2{\beta} -2})}\\
\end{array}
\label{cz2.5ghju48156}
\end{equation}
for all $t\in(0,T_{max})$. Now, we will estimate the right hand of \dref{cz2.5ghju48156}. To this end, firstly, we conclude  from Lemma \ref{lemma41ffgg} that
\begin{equation}
\begin{array}{rl}
&\disp{\disp\int_{\partial\Omega}\frac{\partial |\nabla {v}|^2}{\partial\nu} |\nabla {v}|^{2\beta-2} }\\
\leq&\disp{C_\Omega\disp\int_{\partial\Omega} |\nabla {v}|^{2\beta} }\\
=&\disp{C_\Omega\||\nabla {v}|^{\beta}\|^2_{L^2(\partial\Omega)}.}\\
\end{array}
\label{cz2.57151hhkkhhgg}
\end{equation}
Let us take $r\in(0,\frac{1}{2})$. By the embedding $W^{r+\frac{1}{2},2}(\Omega)\hookrightarrow L^2(\partial\Omega)$ is compact, we have
\begin{equation}
\begin{array}{rl}
&\disp{\| |\nabla {v}|^{\beta}\|^2_{L^2{(\partial\Omega})}\leq C_1\||\nabla {v}|^{\beta}\|^2_{W^{r+\frac{1}{2},2}(\Omega)}.}\\
\end{array}
\label{cz2.57151}
\end{equation}
In order to apply Lemma \ref{lemma41ffgg} to estimate the right-hand side of \dref{cz2.57151}, let us pick $a\in(0,1)$ satisfying
$$a=\frac{\frac{1}{4}+\frac{\beta}{2}+\frac{\gamma}{N}-\frac{1}{2}}{\frac{1}{N}+\frac{\beta}{2}-\frac{1}{2}}.$$
Noting that $\gamma\in(0,\frac{1}{2})$ and $\beta>1$ imply that $\gamma+\frac{1}{2}\leq a<1$, we see from the fractional Gagliardo--Nirenberg inequality  and boundedness of $ |\nabla {v}|^2$ (see Lemma \ref{wsdelemma45}) that
\begin{equation}
\begin{array}{rl}
&\disp{\| |\nabla {v}|^{\beta}\|^2_{W^{r+\frac{1}{2},2}(\Omega)}}
\\
\leq&\disp{c_0\|\nabla |\nabla {v}|^{\beta}\|^a_{L^2(\Omega)}\| |\nabla {v}|^\beta\|^{1-a}_{L^{\frac{2}{\beta}}(\Omega)}+c'_0\| |\nabla {v}|^\beta\|_{L^{\frac{2}{\beta}}(\Omega)}}\\
\leq&\disp{C_2\|\nabla |\nabla {v}|^{\beta}\|^a_{L^2(\Omega)}+C_2.}\\
\end{array}
\label{vvggcz2.57151}
\end{equation}
Combining \dref{cz2.57151hhkkhhgg} and \dref{cz2.57151} with \dref{vvggcz2.57151}, we obtain
\begin{equation}
\begin{array}{rl}
&\disp{\disp\int_{\partial\Omega}\frac{\partial |\nabla {v}|^2}{\partial\nu} |\nabla {v}|^{2\beta-2} \leq C_3\|\nabla |\nabla {v}|^{\beta}\|^a_{L^2(\Omega)}+C_3.}\\
\end{array}
\label{cz2.57151hhkkhhggyyxx}
\end{equation}
On the other hand, by
 $|\Delta {v}| \leq\sqrt{N}|D^2{v}|$ and the Young inequality, we can get
\begin{equation}
\begin{array}{rl}
\disp\int_\Omega {u}v |\nabla {v}|^{2{\beta} -2}\Delta {v}\leq&\disp{\sqrt{N}\|v_0\|_{L^\infty(\Omega)}\int_\Omega {u} |\nabla {v}|^{2{\beta} -2}|D^2{v}|}
\\
\leq&\disp{\frac{1}{4}\int_\Omega  |\nabla {v}|^{2{\beta} -2}|D^2{v}|^2+N\|v_0\|_{L^\infty(\Omega)}^2\int_\Omega u^2 |\nabla {v}|^{2{\beta} -2}.}\\
\end{array}
\label{cz2.5ghju48hjuikl1}
\end{equation}
Next,
due to the Cauchy--Schwarz inequality, we have
\begin{equation}
\begin{array}{rl}
\disp\int_\Omega {u}v\nabla {v}\cdot\nabla( |\nabla {v}|^{2{\beta} -2})= &\disp{({\beta} -1)\int_\Omega {u}
v |\nabla {v}|^{2({\beta} -2)}\nabla {v}\cdot
\nabla |\nabla {v}|^{2}}\\
\leq &\disp{\frac{{\beta} -1}{8}\int_{\Omega} |\nabla {v}|^{2{\beta} -4}\left|\nabla |\nabla {v}|^{2}\right|^2+2({\beta} -1)\|v_0\|_{L^\infty(\Omega)}^2
\int_\Omega |{u}|^2 |\nabla {v}|^{2{\beta} -2}}\\
\leq &\disp{\frac{({\beta} -1)}{2{{\beta} ^2}}\int_{\Omega}\left|\nabla |\nabla {v}|^{{\beta} }\right|^2+2({\beta} -1)\|v_0\|_{L^\infty(\Omega)}^2
\int_\Omega |{u}|^2 |\nabla {v}|^{2{\beta} -2}.}\\
\end{array}
\label{cz2.5ghju4ghjuk81}
\end{equation}
Now,
collecting \dref{cz2.5ghju48156}, \dref{cz2.57151hhkkhhggyyxx}--\dref{cz2.5ghju4ghjuk81} and using the Young inequality
yields
\begin{equation}
\label{hjui909klopjiddfff115}
\begin{array}{rl}
&\disp{\frac{1}{2\beta}\frac{d}{dt}\|\nabla {v}\|^{{{2{\beta} }}}_{L^{{2{\beta} }}(\Omega)}+\frac{({\beta} -1)}{{{\beta} }}\int_{\Omega}\left|\nabla |\nabla {v}|^{{\beta} }\right|^2+\frac{3} {4}\int_\Omega  |\nabla {v}|^{2{\beta} -2}|D^2{v}|^2}\\
\leq&\disp{C_4\int_\Omega u^2 |\nabla {v}|^{2{\beta} -2}+C_4~~\mbox{for all}~~ t\in(0,T_{max}).}\\
\end{array}
\end{equation}


Next, in light of the Young inequality and using  Lemma \ref{ghyuushhuuusdeedrfe116lemma70hhjj}, we derive
\begin{equation}\label{hjuidderf909kddertghhlopji115}
\begin{array}{rl}
C_4\disp\int_\Omega {u}^2 |\nabla {v}|^{2\beta-2}\leq&\disp{\frac{1}{8\kappa_0}\int_\Omega  |\nabla {v}|^{2\beta+2}+C_5\int_\Omega {u}^{\beta+1}}\\
\leq&\disp{\frac{1}{8}\int_\Omega  |\nabla {v}|^{2\beta-2}|D^2{v}|^2+C_5\int_\Omega {u}^{\beta+1}+C_6~~\mbox{for all}~~ t\in(0,T_{max}),}\\
\end{array}
\end{equation}
where $\kappa_0$ is the same as \dref{aqwswdggyuuudaassffssff3.10deerfgghhjuuloollgghhhyhh}.
Inserting \dref{hjuidderf909kddertghhlopji115} into \dref{hjui909klopjiddfff115}, we conclude that
\begin{equation}
\label{hjui909klopjiddffddfff115}
\begin{array}{rl}
&\disp{\frac{1}{2\beta}\frac{d}{dt}\|\nabla {v}\|^{{{2{\beta} }}}_{L^{{2{\beta} }}(\Omega)}+\frac{({\beta} -1)}{{{\beta} }}\int_{\Omega}\left|\nabla |\nabla {v}|^{{\beta} }\right|^2+\frac{5} {8}\int_\Omega  |\nabla {v}|^{2{\beta} -2}|D^2{v}|^2}\\
\leq&\disp{C_5\int_\Omega {u}^{\beta+1}+C_7~~\mbox{for all}~~ t\in(0,T_{max}).}\\
\end{array}
\end{equation}

Let $p>1.$
Now,
testing  the first equation in \dref{1.ssderrfff1} with  $u^{{p}-1}$ and integrating  over $\Omega$ and using \dref{9162},
 we derive
\begin{equation}
\begin{array}{rl}
&\disp{\frac{1}{{p}}\frac{d}{dt}\|u\|^{{p}}_{L^{{p}}(\Omega)}+({{p}-1})\int_{\Omega}u^{m+{{p}-1}}|\nabla u|^2 }
\\
\leq&\disp{-\chi\int_\Omega \nabla\cdot( u\nabla v)
  u^{{p}-1}  +\mu
\int_\Omega   u^{{p}-1}(u- u^2)  ~~\mbox{for all}~~ t\in(0,T_{max}).}\\
\end{array}
\label{cz2.5114114}
\end{equation}
Next,
integrating by parts to the first term on the right hand side of \dref{cz2.5114114},  using  the Young inequality and Lemma \ref{ghyuushhuuusdeedrfe116lemma70hhjj},
we obtain 
\begin{equation}
\begin{array}{rl}
&\disp{-\chi\int_\Omega \nabla\cdot( u\nabla v)
  u^{{p}-1}  }
\\
\leq&\disp{({{p}-1})\chi\int_\Omega  u^{{{p}-1}}\nabla u\cdot\nabla v }
\\
\leq&\disp{\frac{{{p}-1}}{4}\int_\Omega  u^{m+{{p}-3}}|\nabla u|^2 +({{p}-1})\chi^2\int_\Omega  u^{{{p}+1-m}}|\nabla v|^2}
\\
\leq&\disp{\frac{{{p}-1}}{4}\int_\Omega  u^{m+{{p}-3}}|\nabla u|^2}\\
 &\disp{+\frac{\beta }{\beta +1}\left(\frac{1}{8\kappa_0}(\beta +1)\right)^{-\frac{1}{\beta }}\left[({{p}-1})\chi^2\right]^\frac{\beta +1}{\beta }\int_\Omega  u^{({{p}+1-m})\frac{\beta +1}{\beta }}+\frac{1}{8\kappa_0}\int_\Omega|\nabla v|^{2\beta +2} }
\\
\leq&\disp{\frac{{{p}-1}}{4}\int_\Omega  u^{m+{{p}-3}}|\nabla u|^2+\frac{1}{8}\int_\Omega  |\nabla {v}|^{2\beta-2}|D^2{v}|^2}\\
 &\disp{+\frac{\beta }{\beta +1}\left(\frac{1}{8\kappa_0}(\beta +1)\right)^{-\frac{1}{\beta }}\left[({{p}-1})\chi^2\right]^\frac{\beta +1}{\beta }\int_\Omega  u^{({{p}+1-m})\frac{\beta +1}{\beta }}+C_8~~\mbox{for all}~~ t\in(0,T_{max}),}
\\
\end{array}
\label{cz2.563019114}
\end{equation}
which together with \dref{cz2.5114114} and the Young inequality implies that
\begin{equation}
\begin{array}{rl}
&\disp{\frac{1}{{p}}\frac{d}{dt}\|u\|^{{p}}_{L^{{p}}(\Omega)}+\frac{3({{p}-1})}{4}\int_{\Omega}u^{m+{{p}-1}}|\nabla u|^2+\frac{({\beta} -1)}{{{\beta} }}\int_{\Omega}\left|\nabla |\nabla {v}|^{{\beta} }\right|^2+\frac{5} {8}\int_\Omega  |\nabla {v}|^{2{\beta} -2}|D^2{v}|^2 }
\\
\leq&\disp{\frac{1}{8}\int_\Omega  |\nabla {v}|^{2\beta-2}|D^2{v}|^2+\frac{\beta }{\beta +1}\left(\frac{1}{8\kappa_0}(\beta +1)\right)^{-\frac{1}{\beta }}\left[({{p}-1})\chi^2\right]^\frac{\beta +1}{\beta }\int_\Omega  u^{({{p}+1-m})\frac{\beta +1}{\beta }}}\\
 &\disp{+\mu
\int_\Omega   u^{{p}-1}(u- u^2)+C_8 }\\
\leq&\disp{\frac{1}{8}\int_\Omega  |\nabla {v}|^{2\beta-2}|D^2{v}|^2+\frac{\beta }{\beta +1}\left(\frac{1}{8\kappa_0}(\beta +1)\right)^{-\frac{1}{\beta }}\left[({{p}-1})\chi^2\right]^\frac{\beta +1}{\beta }\int_\Omega  u^{({{p}+1-m})\frac{\beta +1}{\beta }}}\\
 &\disp{-\frac{\mu}{2} \int_\Omega u^{{p}+1}+C_9  ~~\mbox{for all}~~ t\in(0,T_{max}).}\\
\end{array}
\label{cz2.5114ssewwdddggg114}
\end{equation}
Collecting   \dref{hjui909klopjiddffddfff115} and \dref{cz2.5114ssewwdddggg114} yields to
\begin{equation}
\begin{array}{rl}
&\disp{\frac{1}{{p}}\frac{d}{dt}\|u\|^{{p}}_{L^{{p}}(\Omega)}+\frac{1}{2\beta}\frac{d}{dt}\|\nabla {v}\|^{{{2{\beta} }}}_{L^{{2{\beta} }}(\Omega)}+\frac{3({{p}-1})}{4}\int_{\Omega}u^{m+{{p}-1}}|\nabla u|^2}\\
&\disp{+\frac{({\beta} -1)}{{{\beta} }}\int_{\Omega}\left|\nabla |\nabla {v}|^{{\beta} }\right|^2+\frac{1} {2}\int_\Omega  |\nabla {v}|^{2{\beta} -2}|D^2{v}|^2  +\frac{\mu}{2} \int_\Omega u^{{p}+1}}
\\
\leq&\disp{\frac{\beta }{\beta +1}\left(\frac{1}{8\kappa_0}(\beta +1)\right)^{-\frac{1}{\beta }}\left[({{p}-1})\chi^2\right]^\frac{\beta +1}{\beta }\int_\Omega  u^{({{p}+1-m})\frac{\beta +1}{\beta }}}\\
 &+\disp{C_5\int_\Omega {u}^{\beta+1}+C_{10}  ~~\mbox{for all}~~ t\in(0,T_{max}).}\\
\end{array}
\label{cz2.5114ssewwdddertyyddggg114}
\end{equation}
Next, on the other hand, by \dref{cz2.5ghju4ssderfyyuuiioo8156}, we derive that
$$({{p}+1-m})\frac{\beta +1}{\beta }<p+1~~~\mbox{and}~~\beta+1~<p+1.$$
Thus, with the help of the Young inequality, we conclude that
\begin{equation}
\begin{array}{rl}
&\disp{\frac{1}{{p}}\frac{d}{dt}\|u\|^{{p}}_{L^{{p}}(\Omega)}+\frac{1}{2\beta}\frac{d}{dt}\|\nabla {v}\|^{{{2{\beta} }}}_{L^{{2{\beta} }}(\Omega)}+\frac{3({{p}-1})}{4}\int_{\Omega}u^{m+{{p}-1}}|\nabla u|^2}\\
&\disp{+\frac{({\beta} -1)}{{{\beta} }}\int_{\Omega}\left|\nabla |\nabla {v}|^{{\beta} }\right|^2+\frac{1} {2}\int_\Omega  |\nabla {v}|^{2{\beta} -2}|D^2{v}|^2  +\frac{\mu}{4} \int_\Omega u^{{p}+1}}
\\
\leq&\disp{C_{11}  ~~\mbox{for all}~~ t\in(0,T_{max}).}\\
\end{array}
\label{cz2.5114ssewwdddertyyddggddrttyuug114}
\end{equation}
Therefore, letting
  $y:=\disp\int_{\Omega} u ^{p}  +\disp\int_{\Omega} |\nabla {v}|^{2\beta} $ 
in \dref{cz2.5114ssewwdddertyyddggddrttyuug114} yields to
\begin{equation}\label{fgghh77dfvvfdcvfbhjui909klopji115}\frac{d}{dt}y(t)+C_{13}y(t)\leq C_{12}~~\mbox{for all}~~ t\in(0,T_{max}).
\end{equation}
 Thus a standard ODE comparison argument implies
boundedness of $y(t)$ for all $t\in (0, T_{max})$. Clearly, $\|{u}(\cdot, t)\|_{L^p(\Omega)}$ and $\|\nabla {v}(\cdot, t)\|_{L^{2\beta}(\Omega)}$
are bounded for all $t\in (0, T_{max})$.
Obviously, by $m>1$, we have
$\lim_{\beta\rightarrow+\infty}\beta=\lim_{\beta\rightarrow+\infty}\beta+(m-1)(\beta+1)=+\infty,$
hence,
the boundedness of $\|{u}(\cdot, t)\|_{L^{p}(\Omega)}$ and the H\"{o}lder inequality implies the results.
The proof Lemma \ref{lemma45630223} is complete.
\end{proof}
Our main result on global existence and boundedness thereby becomes a straightforward consequence
of Lemmata \ref{lemmadderr45630223}--\ref{lemma45630223} and Lemma \ref{lemma70}.

\begin{lemma}\label{lemmaddfffrsedrffffffgg}
Suppose that the conditions of Theorem  \ref{theorem3} hold.
Let $T\in (0, T_{max})$ and $(u, v)$ be the solution of \dref{1.ssderrfff1}.  Then there exists a constant $C > 0$ independent of $T$ such that the
component $v$ of $(u, v)$ satisfies
\begin{equation}
 \begin{array}{rl}
 \|\nabla v(\cdot,t)\|_{L^\infty(\Omega)}\leq C~~~\mbox{for all}~~t\in (0, T).
\end{array}\label{cz2sedfgg.5g5gghh56789hhjui7ssddd8jj90099}
\end{equation}
\end{lemma}
\begin{proof}
Due to
$\|u(\cdot, t)\|_{L^p(\Omega)}$ is bounded for any large $p$, we infer from the fundamental estimates for Neumann semigroup
(see Lemma 4.1 of \cite{Horstmann791}) or the standard regularity theory of parabolic equation (see e.g. Ladyzenskaja et al.  \cite{Ladyzenskaja710}) that
\dref{cz2sedfgg.5g5gghh56789hhjui7ssddd8jj90099} holds.
\end{proof}
\begin{lemma}\label{lemmasedrffffffgg}
Assume
that $u_0\in C^0(\bar{\Omega})$ and $v_0\in W^{1,\infty}(\bar{\Omega})$ both are nonnegative. Let $T\in(0,T_{max})$ and
$D$ satisfy \dref{91derfff61}--\dref{9162} with $$
 m>\left\{\begin{array}{ll}
1-\frac{\mu}{\chi[1+\lambda_{0}\|v_0\|_{L^\infty(\Omega)}2^{3}]}~~\mbox{if}~~
N\leq2,\\
 1~~~~~~\mbox{if}~~ N\geq3.\\
 \end{array}\right.
$$ 
Then there exists a constant $C > 0$ independent of $T$ such that
\begin{equation}
 \begin{array}{rl}
 \|u(\cdot,t)\|_{L^\infty(\Omega)}\leq C~~~\mbox{for all}~~t\in (0, T).
\end{array}\label{ssddaqwddfffhhhhkkswddaassffssff3.ddfvbb10deerfgghhjuuloollgghhhyhh}
\end{equation}
\end{lemma}
\begin{proof}
Throughout the proof of Lemma \ref{lemmasedrffffffgg}, we use $C_i$ $(i\in \mathbb{N})$ to denote
the different positive constants independent of $p,T$ and $k$ ($k
\in \mathbb{N}).$

Case  $m\geq1:$ For any $p>1,$
multiplying both sides of the first equation in \dref{1.ssderrfff1} by $u^{p-1}$, integrating over $\Omega$, integrating by parts and using the Young inequality and \dref{cz2sedfgg.5g5gghh56789hhjui7ssddd8jj90099}, we derive that
\begin{equation}
\begin{array}{rl}
&\disp{\frac{1}{{p}}\frac{d}{dt}\|u\|^{{p}}_{L^{{p}}(\Omega)}+({{p}-1})\int_{\Omega}u^{m+{{p}-3}}|\nabla u|^2}
\\
=&\disp{-\chi\int_\Omega \nabla\cdot( u\nabla v)
  u^{{p}-1} +
\int_\Omega   u^{{p}-1}(\mu u-\mu u^2) }\\
=&\disp{\chi({p}-1)\int_\Omega  u^{{p}-1}\nabla u\cdot\nabla v
   +
\int_\Omega   u^{{p}-1}(\mu u-\mu u^2) }\\
\leq&\disp{\chi^2({p}-1)C_1\int_\Omega  u^{{p}-1}|\nabla u|+
\int_\Omega   u^{{p}-1}(\mu u-\mu u^2)}\\
\leq&\disp{\frac{({{p}-1})}{4}\int_{\Omega}u^{m+{{p}-3}}|\nabla u|^2+\chi^2({p}-1)C_1^2\int_\Omega u^{{p}+1-m}+
\int_\Omega   u^{{p}-1}(\mu u-\mu u^2)}\\
\leq&\disp{\frac{({{p}-1})}{4}\int_{\Omega}u^{m+{{p}-3}}|\nabla u|^2+\chi^2({p}-1)C_1^2\int_\Omega u^{{p}}+
\int_\Omega   u^{{p}-1}(\mu u-\mu u^2)}\\
\leq&\disp{\frac{({{p}-1})}{4}\int_{\Omega}u^{m+{{p}-3}}|\nabla u|^2+C_2p\int_\Omega u^{{p}}-
\int_\Omega   u^{{p}}-\mu
\int_\Omega   u^{{p}+1}~~ \mbox{for all}~~~  t\in(0,T ),}\\
\end{array}
\label{cz2aasweeettyyiii.5114114}
\end{equation}
where $C_2=C_1^2\chi^2+\mu+1.$
Here we have used the fact that $m\geq1.$
Due to \dref{cz2aasweeettyyiii.5114114}, we conclude that
\begin{equation}
\begin{array}{rl}
&\disp{\frac{d}{dt}\|u\|^{{p}}_{L^{{p}}(\Omega)}+\int_\Omega   u^{{p}}+C_3\int_{\Omega}|\nabla u^{\frac{m+p-1}{2}}|^2+\mu
\int_\Omega   u^{{p}+1}\leq C_2p^2\int_\Omega u^{{p}}~~ \mbox{for all}~~~  t\in(0,T ).}\\
\end{array}
\label{cz2aasweeettyyiii.51sderftgg14114}
\end{equation}
Now,  we let $p_0>\max\{1,m-1\},p:=p_k = 2^k(p_0 + 1-m) + m-1$
and
\begin{equation}M_k =\max\{1,\sup_{t\in(0,T)}\int_{\Omega}u^{p_k}\}~~~\mbox{for}~~k\in \mathbb{N}.
\label{cz2aasweeettyyddrffgghiii.51sderftgg14114}
\end{equation}
Hence, by the Gagliardo--Nirenberg inequality,
\begin{equation}
\begin{array}{rl}
C_2p_k^2\disp\int_\Omega u^{ p_k }=&\disp{C_2p_k^2\|u^{\frac{m+p_k-1}{2}}\|_{L^{\frac{2 p_k }{m+ p_k -1}}(\Omega)}^{\frac{2 p_k }{m+ p_k -1}}}\\
\leq&\disp{C_3 p_k ^2(\|\nabla u^{\frac{m+ p_k -1}{2}}\|_{L^{2}(\Omega)}^{\frac{2 p_k }{m+ p_k -1}\varsigma_1}
\| u^{\frac{m+ p_k -1}{2}}\|_{L^{1}(\Omega)}^{\frac{2 p_k }{m+ p_k -1}(1-\varsigma_1)}+\| u^{\frac{m+ p_k -1}{2}}\|_{L^{1}(\Omega)}^{\frac{2 p_k }{m+ p_k -1}}),}\\
\end{array}
\label{cz2aasweeettyyiii.51sderftgg14114}
\end{equation}
where
$$\frac{2 p_k }{m+ p_k -1}\varsigma_1=\frac{2 p_k }{m+ p_k -1}\frac{N-\frac{N(m+ p_k -1)}{2 p_k }}{1-\frac{N}{2}+N}=\frac{2N( p_k +1-m)}{(N+2)(m+ p_k -1)}<2
$$
and
$$\frac{2 p_k }{m+ p_k -1}(1-\varsigma_1)=\frac{2 p_k }{m+ p_k -1}(1-\frac{N-\frac{N(m+ p_k -1)}{2 p_k }}{1-\frac{N}{2}+N})=2\frac{2 p_k +N(m-1)}{(N+2)(m+ p_k -1)}.
$$
Therefore, an application of the Young inequality yields
\begin{equation}
\begin{array}{rl}
C_2 p_k ^2\disp\int_\Omega u^{ p_k }\leq&\disp{C_4\|\nabla u^{\frac{m+ p_k -1}{2}}\|_{L^{2}(\Omega)}^{2}+C_5 p_k ^{\frac{(N+2)(m+ p_k -1)}{ p_k +(N+1)(m-1)}}
\| u^{\frac{m+ p_k -1}{2}}\|_{L^{1}(\Omega)}^{\frac{2 p_k +N(m-1)}{N(m-1)+m+ p_k -1}}}\\
&\disp{+C_6 p_k ^2\| u^{\frac{m+ p_k -1}{2}}\|_{L^{1}(\Omega)}^{\frac{2 p_k }{m+ p_k -1}}}\\
\leq&\disp{C_3\|\nabla u^{\frac{m+ p_k -1}{2}}\|_{L^{2}(\Omega)}^{2}+C_7 p_k ^{\frac{(N+2)(m+ p_k -1)}{ p_k +(N+1)(m-1)}}
\| u^{\frac{m+ p_k -1}{2}}\|_{L^{1}(\Omega)}^{\frac{2 p_k }{m+ p_k -1}}.}\\
\end{array}
\label{cz2aasweeettyyiiissdff.51sderftsdfffgg14114}
\end{equation}
Here we have use the fact that ${\frac{2 p_k +N(m-1)}{N(m-1)+m+ p_k -1}}\leq\frac{2 p_k }{m+ p_k -1}.$
Thus, in light of $m\geq1,$  by means of \dref{cz2aasweeettyyddrffgghiii.51sderftgg14114}--\dref{cz2aasweeettyyiiissdff.51sderftsdfffgg14114}, 
\begin{equation}
\begin{array}{rl}
\disp\frac{d}{dt}\|u \|^{p_k}_{L^{p_k}(\Omega)}+\int_\Omega  u ^{{p_k}}{}\leq&\disp{C_7 p_k ^{\frac{(N+2)(m+ p_k -1)}{ p_k +(N+1)(m-1)}}
\|   u ^{\frac{m+{p_k}-1}{2}}\|_{L^1(\Omega)}^{\frac{2 p_k }{m+ p_k -1}}}\\
\leq&\disp{\lambda^k
M_{k-1}^{\frac{2 p_k }{m+ p_k -1}}}\\
\leq&\disp{\lambda^k
M_{k-1}^{2}~~\mbox{for all}~~ t\in(0, T)}\\
\end{array}
\label{zjscz2.5297x9630111rrd67ddfff512df515}
\end{equation}
with some $\lambda> 1.$
Here we have use the fact that
$${\frac{(N+2)(m+ p_k -1)}{ p_k +(N+1)(m-1)}}=\frac{2^k(p_0+1-m)(N+2)+2(N+2)(m-1)}{2^k(p_0+1-m)+(N+2)(m-1)}\leq N+2$$
and
$${\frac{2 p_k }{m+ p_k -1}}\leq{\frac{2 (p_k+m-1) }{m+ p_k -1}}=2.$$
Integrating \dref{zjscz2.5297x9630111rrd67ddfff512df515} over $(0, t)$ with $t\in(0, T)$, we derive
\begin{equation}
\begin{array}{rl}
\disp{\int_{\Omega}u^{p_k}(x,t)\leq\max\{\int_\Omega  u ^{{p_k}}_0{},\lambda^k
M_{k-1}^{2}\}~~\mbox{for all}~~ t\in(0, T).}\\
\end{array}
\label{zjscz2.5297x9630111rrd67ddfff512ddfggghhhdf515}
\end{equation}
If $\int_{\Omega}u^{p_k}(x,t) \leq \int_\Omega  u ^{{p_k}}_0{}$ for any large $k\in \mathbb{N},$
then we obtain \dref{ssddaqwddfffhhhhkkswddaassffssff3.ddfvbb10deerfgghhjuuloollgghhhyhh} directly.
Otherwise,
by a straightforward induction, we have
\begin{equation}
\begin{array}{rl}
\disp\int_{\Omega} u^{p_k} {}\leq&\disp{
\lambda^k
(\lambda^{k-1}M_{k-2}^{2})^{2}}\\
=&\disp{\lambda^{k+2(k-1)}M_{k-2}^{2^2}}\\
\leq&\disp{\lambda^{k+\Sigma_{j=2}^k(j-1)}M_{0}^{2^k}.}\\
\end{array}
\label{cz2.56303hhyy890678789ty4tt8890013378}
\end{equation}
 Taking $p_k$-th
roots on both sides of \dref{cz2.56303hhyy890678789ty4tt8890013378}, using the fact that 
$\ln(1 + z)\leq z$ for
all $z\geq 0$,   we can easily get \dref{ssddaqwddfffhhhhkkswddaassffssff3.ddfvbb10deerfgghhjuuloollgghhhyhh}.

Case $N\leq2$ and $1-\frac{\mu}{\chi[1+\lambda_{0}\|v_0\|_{L^\infty(\Omega)}2^{3}]}<m<1$: Due to Lemma \ref{lemmadderr45630223}, we may choose
\begin{equation}
\tilde{p}_0:=1+30\frac{\mu}{\chi[1+\lambda_{0}\|v_0\|_{L^\infty(\Omega)}2^{3}]}
\label{ssdd999zjscz2.52ssdffssderrfgfgff97x96302222114}
\end{equation}
 such that
%
%
%
\begin{equation}
\int_{\Omega}u^{\tilde{p}_0}(x,t) \leq C_9 ~~~\mbox{for all}~~ t\in(0,T_{max}).
\label{999zjscz2.52ssdffssderrfgfgff97x96302222114}
\end{equation}
Next,
testing the first equation in \dref{1.ssderrfff1} by $u^{p-1}$, integrating over $\Omega$, integrating by parts and applying the Young inequality and \dref{cz2sedfgg.5g5gghh56789hhjui7ssddd8jj90099}, we derive that
\begin{equation}
\begin{array}{rl}
&\disp{\frac{1}{{p}}\frac{d}{dt}\|u\|^{{p}}_{L^{{p}}(\Omega)}+({{p}-1})\int_{\Omega}u^{m+{{p}-3}}|\nabla u|^2}
\\
=&\disp{-\chi\int_\Omega \nabla\cdot( u\nabla v)
  u^{{p}-1} +
\int_\Omega   u^{{p}-1}(\mu u-\mu u^2) }\\
\leq&\disp{\chi({p}-1)C_1\int_\Omega  u^{{p}-1}|\nabla u|+
\int_\Omega   u^{{p}-1}(\mu u-\mu u^2)}\\
\leq&\disp{\frac{({{p}-1})}{4}\int_{\Omega}u^{m+{{p}-3}}|\nabla u|^2+\chi^2({p}-1)C_1^2\int_\Omega u^{{p}+1-m}+
\int_\Omega   u^{{p}-1}(\mu u-\mu u^2)}\\
\leq&\disp{\frac{({{p}-1})}{4}\int_{\Omega}u^{m+{{p}-3}}|\nabla u|^2+C_{10}p\int_\Omega u^{p+1-m}-
\int_\Omega   u^{{p}}-\mu
\int_\Omega   u^{{p}+1}~~ \mbox{for all}~~~  t\in(0,T ),}\\
\end{array}
\label{999cz2aasweeettyyiii.5114ddffgg114}
\end{equation}
where $C_{10}=C_1^2\chi^2+\mu+1.$ Here we have use the fact that $1-\frac{\mu}{\chi[1+\lambda_{0}\|v_0\|_{L^\infty(\Omega)}2^{3}]}<m<1$.
Therefore, \dref{999cz2aasweeettyyiii.5114ddffgg114} yields to
\begin{equation}
\begin{array}{rl}
&\disp{\frac{d}{dt}\|u\|^{{p}}_{L^{{p}}(\Omega)}+\int_\Omega   u^{{p}}+C_{11}\int_{\Omega}|\nabla u^{\frac{m+p-1}{2}}|^2+\mu
\int_\Omega   u^{{p}+1}\leq C_{12}p^2\int_\Omega u^{{p}+1-m}~~ \mbox{for all}~~~  t\in(0,T ).}\\
\end{array}
\label{3drftgyyyyyuuucz2aasweeettyyiii.51sderftgg14114}
\end{equation}
Letting  $p:=\tilde{p}_k = 2^k(\tilde{p}_0 + 1-m) + m-1$
and
\begin{equation}\tilde{M}_k =\max\{1,\sup_{t\in(0,T)}\int_{\Omega}u^{\tilde{p}_k}\}~~~\mbox{for}~~k\in \mathbb{N},
\label{999cz2aasweeettyyddrffgghiii.51sderftgg14114}
\end{equation}
where $\tilde{p}_0$ is given by \dref{ssdd999zjscz2.52ssdffssderrfgfgff97x96302222114}.
Thus,  Gagliardo--Nirenberg inequality yields to
\begin{equation}
\begin{array}{rl}
C_{12}\tilde{p}_k^2\disp\int_\Omega u^{ \tilde{p}_k +1-m}=&\disp{C_{12}\tilde{p}_k^2\|u^{\frac{m+\tilde{p}_k-1}{2}}\|_{L^{\frac{2(\tilde{p}_k+1-m) }{m+ \tilde{p}_k -1}}(\Omega)}^{\frac{2(\tilde{p}_k+1-m) }{m+ \tilde{p}_k -1}}}\\
\leq&\disp{C_{13} \tilde{p}_k ^2(\|\nabla u^{\frac{m+ \tilde{p}_k -1}{2}}\|_{L^{2}(\Omega)}^{\frac{2(\tilde{p}_k+1-m) }{m+ \tilde{p}_k -1}\varsigma_2}
\| u^{\frac{m+ \tilde{p}_k -1}{2}}\|_{L^{1}(\Omega)}^{\frac{2(\tilde{p}_k+1-m) }{m+ \tilde{p}_k -1}(1-\varsigma_2)}+\| u^{\frac{m+ \tilde{p}_k -1}{2}}\|_{L^{1}(\Omega)}^{\frac{2(\tilde{p}_k+1-m) }{m+ \tilde{p}_k -1}}),}\\
\end{array}
\label{999cz2aasweeettyyiii.51sderftgg14114}
\end{equation}
where
$$\frac{2(\tilde{p}_k+1-m) }{m+ \tilde{p}_k -1}\varsigma_2=\frac{2(\tilde{p}_k+1-m) }{m+ \tilde{p}_k -1}\frac{2-\frac{2(m+ \tilde{p}_k -1)}{2(\tilde{p}_k+1-m) }}{1-\frac{2}{2}+2}=\frac{ \tilde{p}_k +3(1-m)}{m+ \tilde{p}_k -1}<2
$$
and
$$\frac{2(\tilde{p}_k+1-m) }{m+ \tilde{p}_k -1}(1-\varsigma_2)=\frac{2(\tilde{p}_k+1-m) }{m+ \tilde{p}_k -1}(1-\frac{2-\frac{2(m+ \tilde{p}_k -1)}{2(\tilde{p}_k+1-m) }}{1-\frac{2}{2}+2})=1.
$$
Therefore, in light  of the Young inequality, we conclude that
\begin{equation}
\begin{array}{rl}
C_{12} \tilde{p}_k ^2\disp\int_\Omega u^{ \tilde{p}_k+1-m }\leq&\disp{C_{11}\|\nabla u^{\frac{m+ \tilde{p}_k -1}{2}}\|_{L^{2}(\Omega)}^{2}+C_{14} \tilde{p}_k ^{\frac{4(m+ \tilde{p}_k -1)}{ \tilde{p}_k +5(m-1)}}
\| u^{\frac{m+ \tilde{p}_k -1}{2}}\|_{L^{1}(\Omega)}^{\frac{\tilde{p}_k +5(m-1)}{2(\tilde{p}_k +m-1)}}}\\
&\disp{+C_{15} \tilde{p}_k ^2\| u^{\frac{m+ \tilde{p}_k -1}{2}}\|_{L^{1}(\Omega)}^{\frac{2 (\tilde{p}_k +1-m)}{m+ \tilde{p}_k -1}}}\\
\leq&\disp{C_{11}\|\nabla u^{\frac{m+ \tilde{p}_k -1}{2}}\|_{L^{2}(\Omega)}^{2}+C_{16} \tilde{p}_k ^{\frac{4(m+ \tilde{p}_k -1)}{ \tilde{p}_k +5(m-1)}}
\| u^{\frac{m+ \tilde{p}_k -1}{2}}\|_{L^{1}(\Omega)}^{\frac{2 (\tilde{p}_k+1-m) }{m+ \tilde{p}_k -1}}.}\\
\end{array}
\label{999cz2aasweeettyyiiissdff.51sderftsdfffgg14114}
\end{equation}
Here we have use the fact that $\frac{\tilde{p}_k +5(m-1)}{2(\tilde{p}_k +m-1)}\leq\frac{2 (\tilde{p}_k+1-m) }{m+ \tilde{p}_k -1}$ and $\frac{4(m+ \tilde{p}_k -1)}{ \tilde{p}_k +5(m-1)}\geq2.$

Therefore, in light of $m>1-\frac{\mu}{\chi[1+\lambda_{0}\|v_0\|_{L^\infty(\Omega)}2^{3}]},$  by means of \dref{ssdd999zjscz2.52ssdffssderrfgfgff97x96302222114},
\dref{999cz2aasweeettyyddrffgghiii.51sderftgg14114}--\dref{999cz2aasweeettyyiiissdff.51sderftsdfffgg14114}, 
\begin{equation}
\begin{array}{rl}
\disp\frac{d}{dt}\|u \|^{\tilde{p}_k}_{L^{\tilde{p}_k}(\Omega)}+\int_\Omega  u ^{{\tilde{p}_k}}{}\leq&\disp{C_{16} \tilde{p}_k ^{\frac{4(m+ \tilde{p}_k -1)}{ \tilde{p}_k +5(m-1)}}
\|   u ^{\frac{m+{\tilde{p}_k}-1}{2}}\|_{L^1(\Omega)}^{\frac{2 (\tilde{p}_k+1-m) }{m+ \tilde{p}_k -1}}}\\
\leq&\disp{\tilde{\lambda}^k
M_{k-1}^{\frac{2 (\tilde{p}_k+1-m) }{m+ \tilde{p}_k -1}}~~\mbox{for all}~~ t\in(0, T)}\\
\end{array}
\label{999zjscz2.5297x9630111rrd67ddfff512df515}
\end{equation}
with some $\tilde{\lambda}> 1.$
Here we have use the fact that
$${\frac{4(m+ \tilde{p}_k -1)}{ \tilde{p}_k +5(m-1)}}=4\frac{2^k(p_0+1-m)+2(m-1)}{2^k(p_0+1-m)+6(m-1)}\leq
4\frac{p_0+1-m+2(m-1)}{p_0+1-m+6(m-1)}\leq6$$
and
$${\frac{2 (\tilde{p}_k+1-m) }{m+ \tilde{p}_k -1}}=2{\frac{2^k (\tilde{p}_0+1-m) }{2^k (\tilde{p}_0+1-m)+2(m -1)}}=2(1+{\frac{1-m }{2^k (\tilde{p}_0+1-m)+2(m -1)}}):=\kappa_k.$$
Here we note that $\kappa_k=2(1+\varepsilon_k)$ for $k\geq1$,  where $\varepsilon_k$ satisfies $\varepsilon_k\leq \frac{C_{17}}{2^{k}}$ for all $k$ with some $C_{17}>0$.
Next, we
integrate \dref{999zjscz2.5297x9630111rrd67ddfff512df515} over $(0, t)$ with $t\in(0, T)$, then yields to
\begin{equation}
\begin{array}{rl}
\disp{\int_{\Omega}u^{\tilde{p}_k}(x,t)\leq\max\{\int_\Omega  u ^{{\tilde{p}_k}}_0{},\tilde{\lambda}^k
M_{k-1}^{\frac{2 (\tilde{p}_k+1-m) }{m+ \tilde{p}_k -1}}\}~~\mbox{for all}~~ t\in(0, T).}\\
\end{array}
\label{999zjscz2.5297x9630111rrd67ddfff512ddfggghhhdf515}
\end{equation}
If $\int_{\Omega}u^{\tilde{p}_k}(x,t) \leq \int_\Omega  u ^{{\tilde{p}_k}}_0{}$ for any large $k\in \mathbb{N},$
then we derive  \dref{ssddaqwddfffhhhhkkswddaassffssff3.ddfvbb10deerfgghhjuuloollgghhhyhh} holds.
Otherwise,
by a straightforward induction, we have
\begin{equation}
\begin{array}{rl}
\disp\int_{\Omega} u^{\tilde{p}_k} {}\leq&\disp{
\tilde{\lambda}^{k+\sum_{j=2}^k(j-1)\cdot\prod_{i=j}^k\kappa_i}
\tilde{M}_0^{\prod_{i=1}^k\kappa_i}~~\mbox{for all}~~k\geq1.}\\
\end{array}
\label{ssdd4444cz2.56303hhyy890678789ty4tt8890013378}
\end{equation}
On the other hand, due to the fact that $\ln(1+x)\leq x$ (for all $x\geq0$),
$$\begin{array}{rl}
\disp\prod_{i=j}^k\kappa_i
=&2^{k+1-j}e^{\Sigma_{i=j}^k\ln(1+\varepsilon_j)}\\
\leq&2^{k+1-j}e^{\Sigma_{i=j}^k \varepsilon_j}\\
\leq&2^{k+1-j}e^{C_{17}}~~~\mbox{for all}~~k\geq1~~\mbox{and}~~j=\{1,\ldots,k\}.\\
\end{array}$$
In light of the above inequality, with the help of \dref{ssdd4444cz2.56303hhyy890678789ty4tt8890013378}, we conclude that
\begin{equation}
\begin{array}{rl}
\disp\left(\int_{\Omega} u^{\tilde{p}_k}\right)^{\frac{1}{\tilde{p}_k}} {}\leq&\disp{
\tilde{\lambda}^{\frac{k}{\tilde{p}_k}+\frac{\sum_{j=2}^k(j-1)\cdot\prod_{i=j}^k\kappa_i}{\tilde{p}_k}}
\tilde{M}_0^{\frac{\prod_{i=1}^k\kappa_i}{\tilde{p}_k}}~~\mbox{for all}~~k\geq1,}\\
\end{array}
\label{ssdd4444ddfffcz2.56303hhyy890678789ty4tt8890013378}
\end{equation}
which after taking $k\rightarrow\infty$ readily implies that \dref{ssddaqwddfffhhhhkkswddaassffssff3.ddfvbb10deerfgghhjuuloollgghhhyhh} holds.
\end{proof}
{\bf The proof of Theorem \ref{theorem3}}~
Theorem \ref{theorem3} will be proved if we can show $T_{max}=\infty$. Suppose on contrary that $T_{max}<\infty$.
%
In view of \dref{ssddaqwddfffhhhhkkswddaassffssff3.ddfvbb10deerfgghhjuuloollgghhhyhh}, we apply Lemma \ref{lemma70} to reach a contradiction.
 Hence the  classical solution $(u,v)$ of \dref{1.ssderrfff1} is global in time and bounded.

{\bf Acknowledgement}:
This work is partially supported by the
Natural Science Foundation of Shandong Province of China (No. ZR2016AQ17) and the National Natural
Science Foundation of China (No. 11601215).

\end{document}